\documentclass{amsart}

\usepackage[varg]{txfonts}

\newtheorem{theorem}{Theorem}[section]
\newtheorem{lemma}[theorem]{Lemma}

\newcommand{\NN}{\mathbb{N}}

\newcommand{\RR}{\mathbb{R}}

\newcommand{\st}{\; | \;}

\newcommand{\ex}[1]{\exists #1 \;} 

\title{Uniform distribution and algorithmic randomness}

\author{Jeremy Avigad}

\address{Departments of Philosophy and Mathematical Sciences\\
Carnegie Mellon University\\
Pittsburgh, PA 15213}


\thanks{Work partially supported by NSF grant DMS-1068829 and AFOSR grant FA 9550-12-1-0370. I am grateful to Terence Tao for advice regarding the main result of Section~\ref{ud:kurtz:section}, and to Jason Rute and Edward Dean for helpful discussions. I am also grateful to two anonymous referees for their careful reading, and corrections and suggestions.}

\begin{document}

\begin{abstract}
A seminal theorem due to Weyl \cite{weyl:16} states that if $(a_n)$ is any sequence of distinct integers, then, for almost every $x \in \RR$, the sequence $(a_n x)$ is uniformly distributed modulo one. In particular, for almost every $x$ in the unit interval, the sequence $(a_n x)$ is uniformly distributed modulo one for every \emph{computable} sequence $(a_n)$ of distinct integers. Call such an $x$ \emph{UD random}. Here it is shown that every Schnorr random real is UD random, but there are Kurtz random reals that are not UD random. On the other hand, Weyl's theorem still holds relative to a particular effectively closed null set, so there are UD random reals that are not Kurtz random.
\end{abstract}

\maketitle

\section{Introduction}
\label{introduction:section}

Let $\{ x \}$ denote the fractional part of a real number $x$, and let $\lambda$ denote Lebesgue measure. A sequence $(x_n)_{n \in \NN}$ of real numbers is said to be \emph{uniformly distributed modulo one} if for every interval $I \subseteq [0,1]$, 
\[
 \lim_{n \to \infty} \frac{|\{ i < n \st \{x_i\} \in I \}|}{n} = \lambda(I).
\]
In words, a sequence is uniformly distributed modulo one if the limiting frequency with which it visits any given interval is what one would expect if the elements of the sequence were chosen at random. A remarkable theorem by Hermann Weyl \cite{weyl:16} states the following:
\begin{theorem}
If $(a_n)$ is any sequence of distinct integers, then for almost every $x$, $(a_n x)$ is uniformly distributed modulo one.
\end{theorem}
\noindent This seminal result lies at the intersection of harmonic analysis, number theory, ergodic theory, and computer science; see, for example, \cite{chazelle:00,harman:98,kuipers:niederreiter:74,rosenblatt:weierdl:95}.
 
Now let us restrict attention to \emph{computable} sequences $(a_n)$ of distinct integers. Since there are only countably many of these, it follows that almost every $x \in [0,1]$ has the property that $(a_n x)$ is uniformly distributed modulo one for each such sequence. Let us say that such an $x$ is \emph{UD random}. Our goal here is to explore the relationship of UD randomness to other senses in which a real number in the unit interval can be said to be ``random.''

Section~\ref{results:section} considers sufficient conditions for UD randomness. In particular, in the terminology of algorithmic randomness, every Schnorr random element of $[0,1]$ is UD random, but there are Kurtz random elements of $[0,1]$ that are not. 

Regarding consequences of UD randomness, recall that a real number $x$ is \emph{normal} to a base $b > 1$ if, when it is expressed in that base, each sequence of $k$ digits occurs with limiting frequency $b^{-k}$. It is not hard to see that a real number $x$ is normal with respect to base $b$ if and only if the sequence $(b^k x)$ is uniformly distributed modulo one, so every UD random real is normal to every base. In fact, more is true: if $x$ is UD random and $(a_n)$ is any computable sequence of distinct integers, then the limiting frequency of occurrences of any $k$-digit block at positions $a_0, a_1, a_2, \ldots$ of $x$ is again $b^{-k}$. Thus a UD random number looks random, at least in some ways. 

In Section~\ref{ud:kurtz:section}, however, I show that there are straightforward ways that a UD random number can look very \emph{nonrandom}. Specifically, consider the set $C$ of elements $x$ in $[0,1]$ such that for every $n$, the $2^{2n}$-th digit in the binary expansion of $x$ is equal to the $2^{2n+1}$-st digit. Then Weyl's theorem holds with respect to the natural measure on $C$, showing that, in particular, almost every element of $C$ is UD random.

\section{A sufficient condition for UD randomness}
\label{results:section}

The subject of algorithmic randomness \cite{downey:hirschfeldt:10,nies:09} aims to characterize different senses in which a real number (or, say, an infinite binary sequence) can be said to be ``random.'' A subset $G$ of $\RR$ is said to be \emph{effectively open} if there are computable sequences $(a_i)_{i \in \NN}$ and $(b_i)_{i \in \NN}$ of rational numbers such that $G = \bigcup_i (a_i, b_i)$. An effectively open subset of the unit interval is obtained by restricting $G$ to $[0,1]$. A sequence $(G_j)_{j \in \NN}$ of effectively open sets is \emph{uniformly effective} if the representing sequences $(a^j_i)_{i \in \NN}$ and $(b^j_i)_{i \in \NN}$ can be computed by a single algorithm with $j$ as a parameter. A \emph{Martin-L\"of test} is a uniformly effective sequence of open sets $(G_j)$ such that for each $j$, $\lambda(G_j) \leq 2^{-j}$. An element $x$ of $[0,1]$ \emph{fails} the Martin-L\"of test $(G_j)$ if it is an element of $\bigcap_j G_j$, and \emph{passes} the test otherwise. An element $x$ of $[0,1]$ is \emph{Martin-L\"of random} if it passes every Martin-L\"of test. In other words, $x$ is Martin-L\"of \emph{non}random if it can be covered, effectively, by arbitrarily small open sets, and is thus contained in an effectively presented null $G_\delta$ set.

One can weaken or strengthen this notion of randomness by restricting or enlarging the class of tests. For example, a \emph{Schnorr test} is a Martin-L\"of test with the additional property that for each $j$, the measure of $G_j$ is computable. This has the effect that when enumerating the intervals that cover the nonrandoms with a set of size at most $2^{-j}$, one knows how much of the set is yet to come. An element $x$ of $[0,1]$ is \emph{Schnorr random} if it passes every Schnorr test.

An element $x$ of $[0,1]$ is \emph{Kurtz random} if it is contained in an effectively closed set of measure $0$, that is, the complement of an effectively open set of measure $1$. Clearly every Martin L\"of random element of $[0,1]$ is Schnorr random, and every Schnorr random element is Kurtz random. The notion of Kurtz randomness is fairly weak. For example, it is not hard to show that if a real number $x$ is weakly 1-generic then it is Kurtz random but its binary digits fail to satisfy the law of large of numbers. (See \cite[Section 2.24]{downey:hirschfeldt:10} for the definition of weak 1-genericity, and \cite[Section 8.11]{downey:hirschfeldt:10} for the facts just mentioned.) In particular, this shows that there are Kurtz random reals that are not UD random.

The main result of this section is this:
\begin{theorem}
\label{main:thm}
 Every Schnorr random element of $[0,1]$ is UD random.
\end{theorem}
The proof of this lemma follows conventional proofs of Weyl's theorem, for example, as presented in \cite{chazelle:00,harman:98,kuipers:niederreiter:74}; we need only make certain constructive and quantitative aspects of the argument explicit. To start with, we need an equivalent formulation of uniform distribution modulo one.

\begin{lemma}
For any sequence $(x_n)$ of real numbers, the following are equivalent:
\begin{enumerate}
 \item $(x_n)$ is uniformly distributed modulo one.
 \item For any continuous function $f: [0,1] \to \RR$, 
\[
 \lim_{n \to \infty} \frac{1}{n} \sum_{j < n} f(x_j \bmod 1) = \int_0^1 f(x) \; dx.
\]
 \item For any integer $h \neq 0$, 
\[
 \lim_{n \to \infty} \frac{1}{n} \sum_{j < n} e^{2\pi i h x_j} = 0.
\]
\end{enumerate}
\end{lemma}
The third property is known as ``the Weyl criterion.'' Roughly, the second follows from the first using Riemann sums to approximate the integral; the third follows immediately from the second; and the first follows from the third, using trigonometric sums to approximate the characteristic function of an interval. For details, see any of the references mentioned above.

It will be convenient to adopt the Vinogradov convention of writing $e(x)$ for $e^{2 \pi i x}$, thus freeing up the symbol $i$ to be used as an index. Note that $\overline{e(x)} = e(-x)$ and $e(x + y) = e(x)e(y)$ for every $x$ and $y$. Theorem~\ref{main:thm} is an immediate consequence of the next lemma, since if $x$ is a Schnorr random element of $[0,1]$ and $(a_j)_{j \in \NN}$ is any computable sequence of distinct integers, the lemma implies that $(a_j x)_{j \in \NN}$ is uniformly distributed modulo one.

\begin{lemma}
\label{main:lemma}
 For each $i$, let $(a^i_j)_{j \in \NN}$ be a sequence of distinct integers, such that $a^i_j$ is computable from $i$ and $j$. Then there is a Schnorr test $C$ such that for every $x$ not in $C$ and every integer $i$, $(a^i_j x)_{j \in \NN}$ is uniformly distributed modulo one.  
\end{lemma}

\begin{proof}
For the moment, fix a sequence $(a_i)$ of distinct integers, and for any $x$ in $[0,1]$ write
\[
S_n(x) = \frac{1}{n}\sum_{j<n} e(a_j x). 
\]
Then we have 
\[
 |S_n(x)|^2 = S_n(x) \overline{S_n(x)} = \frac{1}{n^2} \sum_{j,k < n} e((a_j - a_k)x),
\]
and hence
\[
\int_0^1 |S_n(x)|^2 \; dx = \frac{1}{n^2} \sum_{j,k<n} \int_0^1 e((a_j - a_k)x) \; dx = \frac{1}{n},
\]
since each term in the sum is $0$ except when $j = k$. (A more perspicuous way of carrying out this calculation is to notice that the left hand side is the square $\| S_n \|_2^2$ of the norm of $S_n$ in the Hilbert space $L^2([0,1])$, and the functions $e(a_0 x), e(a_1 x), \ldots$ form an orthornormal set.) Markov's inequality implies that for every $\varepsilon > 0$ we have
\[
 \lambda(\{ x \st |S_n(x)| > \varepsilon\}) =  \lambda(\{ x \st |S_n(x)|^2 > \varepsilon^2\}) \leq \frac{1}{n \varepsilon^2}.
\]
In other words, for $n$ large, $S_n(x)$ is small for most $x$. 

Now fix a doubly-indexed sequence $(a^i_j)$ as in the statement of the lemma, and for each $i$ write
\[
S^i_n(x) = \frac{1}{n}\sum_{j<n} e(a^i_j x). 
\]
Without loss of generality we can assume that for every integer $h > 1$ and $i$ the sequence $(a^i_j h)_{j \in \NN}$ already appears as one of the sequences $(a^{i'}_j)_{j \in \NN}$ for some $i'$, since we can replace the original sequence of sequences with a sequence that includes all such multiples. By the Weyl criterion, then, it suffices to find a Schnorr test $C$ such that for each $x$ not in $C$ we have $\lim_{n \to \infty} S^i_n(x) = 0$. 

A simple calculation shows that for any $m$ such that $n^2 \leq m < (n+1)^2$ we have $|S_m(x)| \leq |S_{n^2}(x)| + 2 / {\sqrt n}$; in other words, between $n^2$ and $(n + 1)^2$, the averages do not change that much. This reduces our task to finding a Schnorr test $C$ such that for each $x$ not in $C$ we have $\lim_{n \to \infty} S^i_{n^2}(x) = 0$.

For each $i$, rational $\varepsilon > 0$, and $n$ define
\[
 A_{i, \varepsilon, m} = \{ x \st \ex{n \geq m} |S^i_{n^2}(x)| > \varepsilon \}.
\]
By the calculation above, we have
\[
 \lambda(A_{i, \varepsilon, m}) \leq \sum_{n \geq m} \lambda(\{x \st |S^i_{n^2}(x)| > \varepsilon \}) \leq \sum_{n \geq m} \frac{1}{n^2 \varepsilon^2}, 
\]
which decreases to $0$ as $m$ approaches infinity. Choose an enumeration $(i_j, \varepsilon_j)$ of all pairs $(i,\varepsilon)$, and for each $k$ and $j$ choose $m_{j,k}$ large enough so that $\lambda(A_{i_j,\varepsilon_j, m_{j,k}}) < 2^{-(j + k + 1)}$. For each $k$, let
\[
 G_k = \bigcup_j A_{i_j,\varepsilon_j,m_{j,k}}.
\]
Then for each $k$, $\lambda(G_k) \leq \sum_j 2^{-(j + k + 1)} = 2^{-k}$. Moreover, the measure of $G_k$ is uniformly computable in $k$, since for each $v$ the measure of $\bigcup_{j \leq v} A_{i_j,\varepsilon_j,m_{j,k}}$ is computable and the measure of $\bigcup_{j > v} A_{i_j,\varepsilon_j,m_{j,k}}$ is at most $\sum_{j > v} 2^{-(j + k + 1)} = 2^{-(k + v)}$. Thus the sequence $(G_k)$ is a Schnorr test, and we need only confirm that it meets the specification above. 

So suppose $x$ passes this test, that is, for some $k$, $x$ is not in $G_k$. Given any $i$ and $\varepsilon > 0$, choose $j$ such that $(i, \varepsilon) = (i_j, \varepsilon_j)$. Since $x$ is not in $G_k$, it is not in $A_{i_j,\varepsilon_j,m_{j,k}}$. This means that for every $n \geq m_{j,k}$, we have $|S^i_{n^2}(x)| \leq \varepsilon$. Since $i$ and $\varepsilon > 0$ were arbitrary, we have that $S^i_{n^2}(x)$ approaches $0$ for every $i$, as required.
\end{proof}

It is well known that given any Schnorr test, one can find a computable real that passes that test. (See the proof of Proposition 7.1.11 in \cite{downey:hirschfeldt:10} or the discussion after Definition 3.5.8 in \cite{nies:09}.) As result, the preceding lemma also has the following nice consequence, to the effect that we can actually \emph{compute} a real number that is UD random with respect to a computable list of computable sequences.

\begin{theorem}
For every uniformly computable sequence of sequences $(a^i_j)_{j \in \NN}$ of distinct integers there is a computable element $x$ of $[0,1]$ that is UD random for these sequences; that is, such that $(a^i_j x)_{j \in \NN}$ is uniformly distributed modulo one for each $i$.  
\end{theorem}

\section{UD random reals with nonrandom properties}
\label{ud:kurtz:section}

The main result of this section is that there are UD random reals that are not Kurtz random. In fact, I prove something stronger. Ignoring the countable set of dyadic rationals, every element $x$ of $[0,1]$ can be uniquely identified with its binary representation, or, equivalently, the element of Cantor space, $2^\NN$, corresponding to the sequence of bits after the (binary) decimal point.  Let $(x)_i$ denote the $i$th binary digit of $x$, and let
\[
 C = \{ x \st \mbox{for every $n$, $(x)_{2^{2n}} = (x)_{2^{2n+1}}$} \},
\]
that is, the set of binary sequences such that for every $n$, the $2^{2n}$-th digit is the same as the $2^{2n+1}$-st digit. (Here and below it is convenient to start the indexing with $1$, so that the first bit of $x$ is $(x)_1$ rather than $(x)_0$.) Clearly $C$ is an effectively closed null set. There is a natural bijection $f$ from $2^\NN$ to $C$ which takes any infinite binary sequence and inserts the appropriate digits at positions $2^{2n+1}$ for each $n$. Let $\mu$ be the ``uniform'' probability measure on $C$ defined by setting $\mu(A) = \lambda(f^{-1}(A))$, where $A$ is any Borel subset of $C$ and $\lambda$ is the uniform measure on $2^\NN$. The following asserts that the analogue of Weyl's theorem holds for this measure on $C$:

\begin{theorem}
\label{main:kurtz:thm}
Let $(a_n)$ be any sequence of distinct integers. Then for every almost every element $x$ with respect to $\mu$, $(a_n x)$ is uniformly distributed modulo one. 
\end{theorem}

This implies that given any countable collection of sequences of distinct integers, in particular all the computable ones, almost every element $x$ of $C$ is UD random with respect to this collection. I am grateful to Terence Tao for suggesting this approach.

We will prove Theorem~\ref{main:kurtz:thm} as follows. The values $\hat \mu(u) = \int_C e(u x) \; d\mu(x)$ are known as the \emph{Fourier-Stieltjes coefficients} of the measure $\mu$, and the proof of Weyl's theorem presented in Section~\ref{results:section} relied on the fact that the Fourier-Stieltjes coefficients $\hat \lambda(u)$ of Lebesgue measure are $0$ when $u \neq 0$. In fact, the next two lemmas show that for the conclusion of Weyl's theorem to hold with respect to a measure $\mu$ it suffices to show that the Fourier-Stieltjes coefficients of $\mu$ approach $0$ sufficiently quickly. After stating this criterion precisely, we will then show that it is satisfied by the measure $\mu$ at hand. 

First, the proof of Weyl's theorem can be modified to obtain the following result due to Davenport, Erd\H{o}s, and LeVeque \cite{davenport:erdos:leveque:63}:
\begin{lemma}
Fix a finite measure $\mu$ on $[0,1]$, and let $s_n(x)$ be a sequence of bounded, $\mu$-integrable functions. Let $S^m_n(x) = \frac{1}{n} \sum_{j <n} e(m s_n(x))$. If the series
\[
 \sum_{n = 1}^\infty \frac{1}{n} S^m_n(x)
\]
converges for every $m \neq 0$, then the sequence $(s_n(x))$ is uniformly distributed modulo one for $\mu$-almost every $x$.
\end{lemma}

The proof can also be found in Kuipers and Niederreiter \cite[Theorem 4.2]{kuipers:niederreiter:74}. (In both places, the result is stated for Lebesgue measure, but the proof establishes the more general result. See also Lo\`eve \cite{loeve:78} for a similar argument in a slightly different context.) Lyons~\cite[Theorem 7]{lyons:86} uses this fact to establish the following strengthening of Weyl's theorem:

\begin{lemma}
\label{lyons:lemma}
If the Fourier-Stieltjes coefficients of $\mu$ have the property that
\[
 \sum_{u=1}^\infty \frac{|\hat \mu(u)|}{u}
\]
converges, then for any sequence $(a_n)$ of distinct integers and $\mu$-almost every $x$, the sequence $(a_n x)$ is uniformly distributed modulo one.
\end{lemma}

(Lyons only states the theorem for strictly increasing sequences $(a_n)$, and shows, more generally, that in the hypothesis one can replace the sequence $(|\hat \mu(u)|)_{u \in \NN}$ by its nonincreasing rearrangement. It is easy to check, however, that the proof only requires that the values of the sequence $(a_n)$ be distinct. See also \cite{parreau:queffelec:09}.) Notice that a mild rate of convergence of $\hat \mu(u)$ to $0$, such as $O(1 / (\log u)^\delta)$ for any $\delta > 0$, is sufficient to meet the criteria of Lemma~\ref{lyons:lemma}. 

We will show that with the choice of $C$ and $\mu$ above, we have $\hat \mu(u) = O(1 / \sqrt{u})$, which suffices to prove Theorem~\ref{main:kurtz:thm}. The intuition behind the theorem is that multiplying $x$ by some number and testing for membership in an interval only enables one to probe ``local'' properties of binary expansion of $x$, and so a sequence $(a_n)$ cannot exploit correlations between bits that are far apart. This intuition is reflected in the fact that the mass of $C$ is distributed sufficiently uniformly throughout the interval $[0,1]$ that periodic functions like $e(u x)$ do not detect significant irregularities. The rest of this section is devoted to carrying out the calculations that back up these intuitions.

Since we are only estimating $\hat \mu(u)$ up to a multiplicative constant, we can assume that $u \geq 16$. Find $n$ such that $2^{2^{2(n+1)}} \leq u < 2^{2^{2(n+2)}}$. Viewing $C$ as a subset of $2^\NN$, write
\[
 C = \{ \pi \sigma b \tau b \rho \st \pi \in P, \sigma \in S, b \in \{0,1\}, \tau \in T, \rho \in R \}
\]
where:
\begin{itemize}
 \item $P$ is the set of binary strings $\pi$ of length $2^{2n+1}$ such that for every $m \leq n$, $(\pi)_{2^{2m}} = (\pi)_{2^{2m+1}}$
 \item $S$ denotes the set of binary strings of length $2^{2n+1} - 1$
 \item $T$ denotes the set of binary strings of length $2^{2n+2} - 1$.
 \item $R$ denotes the set of infinite binary strings $\rho$ such that for every $m > n +1$, $(\rho)_{2^{2m} - 2^{2(n+1)+1}} = (\rho)_{2^{2m+1} - 2^{2(n+1)+1}}$.
\end{itemize}
Let $\alpha$ denote a string of $0$'s of length $2^{2(n+1)+1}$, let $C' = \{ \alpha \rho \st \rho \in R \}$, and for any $r$, let $C' + r$ denote the translate of $C'$ (now viewed as a subset of $[0,1]$) by $r$. For any binary sequence $\beta$, let $n(\beta)$ denote the binary number denoted by $\beta$. Notice that as $\sigma$ ranges over $S$, $n(\sigma)$ ranges from $0$ to $2^{2^{2n+1} - 1} -1$, and as $\tau$ ranges over $T$, $n(\tau)$ ranges from $0$ to $2^{2^{2(n+1)} - 1} - 1$. Thus we can describe $C$ as
\[
\bigcup_{\pi \in P} \bigcup_{i < 2^{2^{2n+1} - 1}} \bigcup_{j < 2^{2^{2(n+1)} - 1}} \bigcup_{b \in \{0,1\}}
\left( C' + \frac{n(\pi)}{2^{2^{2n + 1}}} + \frac{i}{2^{2^{2(n+1)}-1}} + \frac{b}{2^{2^{2(n+1)}}} + \frac{j}{2^{2^{2(n+1)+1}-1}} + \frac{b}{2^{2^{2(n+1)+1}}} \right).
\]
Using the identity $e(x + y) = e(x)e(y)$, the integral we need to evaluate, $\int_C e(u x) d\mu$, is therefore equal to
\begin{multline*}
\sum_{\pi \in P} \sum_{i < 2^{2^{2n+1} - 1}} \sum_{j < 2^{2^{2(n+1)} - 1}} \sum_{b \in \{0,1\}} \\
e\left(\frac{u n(\pi)}{2^{2^{2n + 1}}}\right) \cdot e\left(\frac{ui}{2^{2^{2(n+1)}-1}}\right) \cdot e\left(\frac{uj}{2^{2^{2(n+1)+1}-1}}\right) \cdot e\left(\frac{ub}{2^{2^{2(n+1)}}}\right) \cdot e\left(\frac{ub}{2^{2^{2(n+1)+1}}} \right) \cdot \int_{C'} e(ux) d\mu.
\end{multline*}
This can be rewritten as
\begin{multline*}
\left( \sum_{\pi \in P} e\left(\frac{u n(\pi)}{2^{2^{2n + 1}}}\right) \right) \cdot
\left( \sum_{i < 2^{2^{2n+1} - 1}} e\left(\frac{ui}{2^{2^{2(n+1)}-1}}\right) \right) \cdot
\left( \sum_{j < 2^{2^{2(n+1)} - 1}} e\left(\frac{uj}{2^{2^{2(n+1)+1}-1}}\right) \right) \cdot \\
\left(1 +e\left(\frac{u}{2^{2^{2(n+1)}}}\right) \cdot e\left(\frac{u}{2^{2^{2(n+1)+1}}} \right) \right) \cdot
\int_{C'} e(ux) d\mu.
\end{multline*}
Since $|e(x)| = 1$ for every $x$, by the triangle inequality the absolute value of the first term is at most the cardinality of $P$, which is $2^{2^{2n+1} - n}$. Since $e(\theta i) = e(\theta)^i$ for every $\theta$ and $i$, the second term is a geometric progression. If $2^{2^{2(n+1)}-1}$ divides $u$, each term in the sum is equal to $1$. Otherwise, using the identity $1 + x + x^2 + \ldots x^{m-1} = (x^m - 1) / (x - 1)$ with $x =  e\left(\frac{u}{2^{2^{2(n+1)}-1}}\right)$ and $m = 2^{2^{2n+1} - 1}$, and computing
\[
e\left(\frac{u}{2^{2^{2(n+1)}-1}}\right)^{2^{2^{2n+1} - 1}} = e\left(\frac{u}{2^{(2^{2(n+1)}-1) - (2^{2n+1} - 1)}}\right) = e\left(\frac{u}{2^{2^{2n+1}}}\right),
\]
we have that the second term is equal to 
\[
 \begin{cases}
   2^{2^{2n+1} - 1} & \mbox{if $2^{2^{2(n+1)}-1} \mid u$} \\
   \frac{e\left(\frac{u}{2^{2^{2n+1}}}\right) - 1}{e\left(\frac{u}{2^{2^{2(n+1)}-1}}\right) - 1} &
      \mbox{otherwise.}
 \end{cases}
\]
A similar calculation shows that the third term is equal to 
\[
 \begin{cases}
   2^{2^{2(n+1)} - 1} & \mbox{if $2^{2^{2(n+1)+1}-1} \mid u$} \\
   \frac{e\left(\frac{u}{2^{2^{2(n+1)}}}\right) - 1}{e\left(\frac{u}{2^{2^{2(n+1)+1}-1}}\right) - 1} &
      \mbox{otherwise.}
 \end{cases} 
\]
There is nothing more to say about the fourth term, for the moment. The absolute value of the last term, the integral, is at most 
\[
\mu(C') = \frac{1}{2^{2^{2(n+1)+1} - (n+1)}}, 
\]
though we will have to evaluate this term more precisely later on.

We now consider cases, keeping in mind that $e(x) = 1$ if $x$ is an integer, and $e(x) = -1$ if $x$ is equal to half an odd integer. If $u$ is divisible by $2^{2^{2n + 1}}$ but not $2^{2^{2(n+1)}-1}$, the second term, and hence the product, is $0$.

If $u$ is divisible by $2^{2^{2(n+1)}-1}$ but not $2^{2^{2(n+1)}}$, the second term is equal to $2^{2^{2n+1} - 1}$, the third term is equal to
\[
\frac{e\left(\frac{u}{2^{2^{2(n + 1)}}}\right) - 1}{e\left(\frac{u}{2^{2^{2(n+1)+1}-1}}\right) - 1}
\]
and the fourth term is equal to
\[
1 - e\left(\frac{u}{2^{2^{2(n+1)+1}}} \right)
\]
since $u / 2^{2^{2(n+1)}}$ is equal to half an odd integer. But the product of the third and fourth terms is
\[
 - \left(e\left(\frac{u}{2^{2^{2(n + 1)+1}}}\right) - 1\right) \cdot \left(e\left(\frac{u}{2^{2^{(2n+1)}}}\right) + 1 \right)
\]
since the denominator of the third term is of the form $z^2 - 1$, where the fourth term is $1 - z$. In absolute value, then, the product of these terms is at most $4$. Thus the absolute value of $\int_C e(ux) \; d\mu$ is at most
\[
2^{2^{2n+1} - n} \cdot 2^{2^{2n+1} - 1} \cdot 4 \cdot \frac{1}{2^{2^{2(n+1)+1} - (n+1)}} = \frac{4}{2^{2^{2(n+1)}}} < \frac{4}{\sqrt u},
\]
since $u > 2^{2^{2(n+2)}}$.

If $u$ is divisible by $2^{2^{2(n+1)}}$ but not $2^{2^{2(n+1)+1}-1}$, then the third term, and hence the product, is equal to $0$. If $u$ is divisible by $2^{2^{2(n+1)+1}-1}$ but not $2^{2^{2(n+1)+1}}$, then $u / 2^{2^{2(n+1)}}$ is an integer and $u / 2^{2^{2(n+1)+1}}$ is a half integer, in which case the fourth term, and hence the product, is equal to $0$. 

Hence we are left with two cases: first, where $u$ is not divisible by $2^{2^{2n + 1}}$, and, second, where $u$ is divisible by $2^{2^{2(n+1)+1}}$. In both cases, we need to further expand the value of $\int_{C'} e(ux) \; d\mu$. Write 
\[    
C' = \{ \alpha \sigma' b \tau' b \rho' \st \sigma' \in S', \tau' \in T', b \in \{0, 1\}, \rho' \in R' \},
\]
where, as before, $\alpha$ is a string of $0$'s of length $2^{2(n+1)+1}$, and $S'$, $T'$, $R'$ are obtained by replacing $n$ by $n+1$ in the definitions of $S$, $T$, and $R$, respectively. Define $C''$ analogously. As above, $\int_{C'} e(ux) \; d\mu$ evaluates to
\begin{multline*}
\left( \sum_{i' < 2^{2^{2(n+1)+1} - 1}} e\left(\frac{ui'}{2^{2^{2(n+1)+1}-1}}\right) \right) \cdot
\left( \sum_{j' < 2^{2^{2(n+2)} - 1}} e\left(\frac{uj'}{2^{2^{2(n+2)}-1}}\right) \right) \cdot \\
\left(1 +e\left(\frac{u}{2^{2^{2(n+2)}}}\right) \cdot e\left(\frac{u}{2^{2^{2(n+2)+1}}} \right) \right) \cdot
\int_{C''} e(ux) d\mu.
\end{multline*}
In the case where $u$ is divisible by $2^{2^{2(n+1)+1}}$, the argument proceeds as before. Since, by assumption, $u < 2^{2^{2(n+2)}}$, we know that $u$ is not divisible by $2^{2^{2(n+2)}}$, and the argument terminates after the second new case.

Thus we are left with the situation where $u$ is not divisible by $2^{2^{2n+1}}$, in which case $\int_C e(ux) \; d\mu$ evaluates to the following:
\begin{multline*}
\left( \sum_{\pi \in P} e\left(\frac{u n(\pi)}{2^{2^{2n + 1}}}\right) \right) \cdot \\
\left(\frac{e\left(\frac{u}{2^{2^{2n+1}}}\right) - 1}{e\left(\frac{u}{2^{2^{2(n+1)}-1}}\right) - 1} \right) \cdot
\left(\frac{e\left(\frac{u}{2^{2^{2(n+1)}}}\right) - 1}{e\left(\frac{u}{2^{2^{2(n+1)+1}}}\right) - 1} \right) \cdot
\left(1 +e\left(\frac{u}{2^{2^{2(n+1)}}}\right) \cdot e\left(\frac{u}{2^{2^{2(n+1)+1}}} \right) \right) \cdot \\
\left(\frac{e\left(\frac{u}{2^{2^{2(n+1)+1}}}\right) - 1}{e\left(\frac{u}{2^{2^{2(n+2)}-1}}\right) - 1} \right) \cdot
\left(\frac{e\left(\frac{u}{2^{2^{2(n+2)}}}\right) - 1}{e\left(\frac{u}{2^{2^{2(n+2)+1}}}\right) - 1} \right)\cdot
\left(1 +e\left(\frac{u}{2^{2^{2(n+2)}}}\right) \cdot e\left(\frac{u}{2^{2^{2(n+2)+1}}} \right) \right) \cdot \\
\int_{C''} e(ux) d\mu.
\end{multline*}
Notice that the denominator of the second fraction is equal to the numerator of the third. Also, the denominator of the first fraction is of the form $z - 1$, where the numerator of the second fraction is $z^2 - 1$. Simplifying the quotient then leaves a multiplicand with absolute value at most 2; and similarly for the denominator of the third fraction and the numerator of the fourth. The numerator of the first fraction has absolute value at most 2, as do the fourth and seventh terms in the product. Finally, the absolute value of the last term is at most 
\[
\mu(C'') = \frac{1}{2^{2^{2(n+2)+1} - (n+2)}}.
\]
Thus the absolute value of $\int_C e(ux) \; d\mu$ is at most
\[
2^{2^{2n+1} - n} \cdot 2^5 \cdot \frac{1}{\left| e\left(\frac{u}{2^{2^{2(n+2)+1}}}\right) - 1\right|} \cdot \frac{1}{2^{2^{2(n+2)+1} - (n+2)}} = \frac{2^{2^{2n+1} +7}}{2^{2^{2(n+2)+1}}} \cdot \frac{1}{\left| e\left(\frac{u}{2^{2^{2(n+2)+1}}}\right) - 1 \right|}.
\]
All we need to do now is to get a lower bound on the denominator of the last fraction, and then use the lower bound $2^{2^{2(n+1)}}$ on $u$ to show that the resulting expression is $O(1 / \sqrt u)$. When $\theta$ is any real value satisfying $|\theta| \leq 1$, a Taylor-series approximation shows
\[
 | e^{i\theta} - 1 | = \sqrt {2|1 - \cos \theta|} \geq \theta / \sqrt 2.
\]
Thus when $x < 1 / 2\pi$, we have $|e(x) - 1| \geq 2\pi x / \sqrt 2 \geq x$. As a result, the absolute value of $\int_C e(ux)\; d\mu$ is at most
\[
\frac{2^{2^{2n+1} +7}}{2^{2^{2(n+2)+1}}} \cdot \frac{2^{2^{2(n+2)+1}}}{u} =  \frac{2^7 \cdot 2^{2^{2n+1}}}{u} \leq
\frac{2^7}{\sqrt u},
\]
since we have $u \geq 2^{2^{2n +2}}$.

\section{Final comments}

The goal here has been to explore the extent to which real numbers satisfying the conclusion of Weyl's theorem for computable sequences look ``random.'' Generally speaking, the failure of Kurtz randomness is striking, since it can be expressed in terms of observable properties. That is, the statement that a real number $x$ is in an effectively closed set $C$ is a universal property, which means that if $x$ is \emph{not} in $C$, then one can verify this fact by carrying out a finite computation on finitely many bits of $x$. Thus, Section~\ref{ud:kurtz:section} shows that there are real numbers $x$ that are UD random, but, at any finite level of accuracy, can be seen to satisfy a distinctly nonrandom property.

UD random reals may fail to be random in other striking ways. For example, one can ask:
\begin{quote}
 Is there a UD random $x$ such that every initial segment of the binary representation of $x$ has at least as many 1's as 0's? More generally, does Weyl's theorem hold relative to a suitable measure on this set?
\end{quote}
I suspect that the answer to these questions is ``yes.''

It is interesting to compare the notion of UD randomness to the notion of Church stochasticity. Roughly speaking, a real number $x$ is Church stochastic if the zeros in any subsequence of digits of the binary representation of $x$ obtained by a computable ``selection procedure'' has a limiting density of one half. This notion is, in a sense, orthogonal to UD randomness: by a suitable choice of the sequence $(a_n)$, a test for UD randomness can sample bits in any order; but, to compensate, a computable selection procedure is allowed to see the previous bits of $x$ before deciding whether or not to select the next bit. Ville's theorem \cite{ville:39} (see also \cite[Section 6.5]{downey:hirschfeldt:10}) shows that there are real numbers $x$ that are Church stochastic (in fact, with respect to any countable collection of selection procedures) but have the property that every initial segment of the binary representation of $x$ has at least as many $1$'s as $0$'s. Wang \cite{wang:96b} has shown that there are numbers $x$ that are Schnorr random but not Church stochastic (see also \cite[Section 8.4]{downey:hirschfeldt:10}). By Theorem~\ref{main:thm}, this implies that UD randomness does not imply Church stochasticity. One can ask about the converse direction:
\begin{quote}
 Is there a real number $x$ that is Church stochastic but not UD random?
\end{quote}
See \cite{merkle:03} for various notions of stochasticity and some of their properties.

Ultimately, I hope the results here suggest that it is interesting and worthwhile to study the relationships between notions of randomness that are implicit in ordinary mathematical theorems.


\end{document}